\documentclass[11pt]{amsart}

\title[Twisted book decompositions and the Goeritz groups]
{Twisted book decompositions and the Goeritz groups}

\author{Daiki Iguchi}
\address{
Department of Mathematics \newline
\indent Hiroshima University, 1-3-1 Kagamiyama, Higashi-Hiroshima, 739-8526, Japan}
\email{m186064@hiroshima-u.ac.jp}

\author{Yuya Koda}
\address{
Department of Mathematics \newline
\indent Hiroshima University, 1-3-1 Kagamiyama, Higashi-Hiroshima, 739-8526, Japan}
\email{ykoda@hiroshima-u.ac.jp}

\usepackage{amsthm}
\usepackage{mathrsfs}
\usepackage{latexsym}
\usepackage[dvips]{graphicx}
\usepackage[dvips]{psfrag}
\usepackage[dvips]{color}
\usepackage{xypic}
\usepackage[abs]{overpic}
\usepackage{caption}
\usepackage[all]{xy}
\usepackage[normalem]{ulem}

\theoremstyle{plain}
\newtheorem*{theorem*}{Theorem}
\newtheorem*{lemma*} {Lemma}
\newtheorem*{corollary*} {Corollary}
\newtheorem*{proposition*}{Proposition}
\newtheorem*{conjecture*}{Conjecture}
\newtheorem{theorem}{Theorem}[section]
\newtheorem{lemma}[theorem]{Lemma}
\newtheorem{corollary}[theorem]{Corollary}
\newtheorem{proposition}[theorem]{Proposition}

\theoremstyle{remark}

\newtheorem*{claim*}{Claim}

\newtheorem*{remark}{Remark}

\theoremstyle{definition}

\newtheoremstyle{citing}
  {}
  {}
  {\itshape}
  {}
  {\bfseries}
  {.}
  {.5em}
  {\thmnote{#3}}

\theoremstyle{citing}
\newtheorem*{citingtheorem}{} 

\newcommand{\Real}{\mathbb{R}}

\newcommand{\Rational}{\mathbb{Q}}
\newcommand{\MCG}{\mathrm{MCG}}

\newcommand{\Aut}{\mathrm{Aut}}

\newcommand{\id}{\mathrm{id}}
\newcommand{\Nbd}{\operatorname{Nbd}}
\newcommand{\Cl}{\operatorname{Cl}}
\newcommand{\Int}{\operatorname{Int}}
\newcommand{\Image}{\operatorname{Im}}
\newcommand{\ML}{\mathcal{ML}}



\begin{document}

\maketitle

\begin{abstract}
We consider the Goeritz groups of the Heegaard splittings induced from 
twisted book decompositions. 
We show that there exist Heegaard splittings of distance $2$ that have 
the infinite-order mapping class groups whereas 
that are not induced from open book decompositions. 
Explicit computation of those mapping class groups are given. 
\end{abstract}

\vspace{1em}

\begin{small}
\hspace{2em}  \textbf{2010 Mathematics Subject Classification}: 
57N10; 57M60


\hspace{2em} 
\textbf{Keywords}:
3-manifold, Heegaard splitting, mapping class group, distance. 
\end{small}

\section*{Introduction}

It is well known that every closed orientable $3$-manifold $M$ is the result of taking 
two copies $H_1$, $H_2$ of a handlebody and gluing them along their boundaries.
Such a decomposition $M = H_1 \cup_\Sigma H_2$ is called 
a {\it Heegaard splitting} for $M$. 
The surface $\Sigma$ here is called the {\it Heegaard surface} of the splitting, 
and the genus of $\Sigma$ is called its {\it genus}.
In \cite{Hem01}, Hempel introduced a measure of the complexity of 
a Heegaard splitting called the {\it distance} of the splitting. 
Roughly speaking, this is the distance between the sets of meridian disks of $H_1$ and $H_2$ 
 in the {\it curve graph} $\mathcal{C} (\Sigma)$ 
of the Heegaard surface $\Sigma$. 

The {\it mapping class group}, or the {\it Goeritz group}, of a Heegaard splitting for a $3$-manifold 
is the group of isotopy classes of orientation-preserving automorphisms (self-homeomorphisms) of the manifold 
that preserve each of the two handlebodies of the splitting setwise.
We note that the Goeritz group of a Heegaard splitting is a subgroup of the mapping class group of the Heegaard surface.

Concerning the structure of the Goeritz groups, Minsky asked in \cite{Gor07} 
when the Goeritz group of a Heegaard splitting is finite, finitely generated, or finitely presented, respectively.  
The distance of Heegaard splittings gives a nice way to describe those nature of the Goeritz groups. 
 In \cite{Nam07}, Namazi showed that the Goeritz group 
 is a finite group if a Heegaard splitting has a sufficiently high distance. 
This result was improved by Johnson \cite{Joh10} showing the same consequence 
when the distance of the splitting is at least $4$. 
On the contrary, it is an easy fact that the Goeritz group 
 is always an infinite group when the distance of the Heegaard splitting is at most one 
 (see e.g. Johnson-Rubinstein \cite{JR13} or Namazi \cite{Nam07}).
In this case, there have been many efforts to find finite generating sets or presentations of 
the Goeritz groups. 
For example, the sequence of works  \cite{Goe33, Sch04, Akb08, Cho08, Cho13, CK14, CK15, CK16, CK19} 
by many authors 
completed to give a finite presentation of the Goeritz group of every genus-$2$ Heegaard splitting 
of distance $0$. 
Recently, Freedman-Scharlemann \cite{FS18} gave a finite generating set of the genus-$3$ Heegaard splitting of the $3$-sphere. 
For the higher genus Heegaard splittings of the 3-sphere, 
the problem of existence of finite generating sets of the Goeritz groups 
still remains open. 
For other works on finite generating sets of Goeritz groups, see \cite{Joh11a, Joh11b, CKS16}. 


In this paper, we concern the Goeritz groups of {\it strongly-irreducible} 
(that is, distance at least $2$) Heegaard splittings. 
There are few isolated examples that are known. 
First, we think of a natural question: how can the Goeritz group be ``small" 
fixing the genus and the distance of the splitting.  
In Section~\ref{sec:The Goeritz groups of keen Heegaard splittings}, 
we consider finiteness properties of the Goeritz groups of {\it keen} Heegaard splittings 
(see Proposition~\ref{thm:The Goeritz groups of keen Heegaard splittings}). 
As a direct corollary, we get the following: 
\begin{citingtheorem}[Corollary~\ref{cor:direct corollary of IJK18}]
For any $g \geq 3$ and $n \geq 2$, there exists a genus-$g$ Heegaard splitting 
of distance $n$ whose 
Goeritz group is either a finite cyclic group or a finite dihedral group. 
\end{citingtheorem}

Roughly speaking, it is believed that the ``majority" of 
the Heegaard splittings of distance $2$ or $3$ have 
the Goeritz groups of at most finite orders.  
One typical example of a ``minority" here is constructed by 
using an open book decomposition with a monodromy of infinite order, 
see for instance the preprint Johnson \cite{Joh11d}. 
In fact, this construction gives a distance-$2$ Heegaard splitting 
whose Goeritz group is an infinite groups. 
Since the Heegaard splitting induced from an open book decomposition 
admits the ``accidental" symmetry coming from the rotation around the binding, 
we might wonder wether this type of Heegaard splittings is the only ``minority". 

In the main part of the paper, we focus on 
the Heegaard splittings induced from {\it twisted book decompositions}, 
which are first studied in Johnson-Rubinstein \cite{JR13}. 
Here is a brief construction 
(see Sections~\ref{sec:Handlebodies as interval bundles}--\ref{sec:The Goeritz groups of the Heegaard splittings induced from twisted book decompositions} 
for the detailed definitions). 
Let $F$ be a compact non-orientable surface of negative Euler characteristic 
with a single boundary component,  
let $\pi: H \to F$ be the orientable $I$-bundle with the {\it binding} $b \subset \partial H =: \Sigma$. 
Let $M = H_1 \cup_\Sigma H_2$ be the Heegaard splitting obtained 
by gluing $H$ to a copy of itself via an automorphism $\varphi$ of $\Sigma$ that preserves $b$. 
It is easy to see that the distance of such a Heeegaard splitting is at most $4$. 
We compute the Goeritz group of $M = H_1 \cup_\Sigma H_2$ in the following two cases. 

The first case is that the gluing map $\varphi$ is particularly ``simple". 
\begin{citingtheorem}[Theorem~\ref{thm:the Goeritz groups of distance-2 Heegaard splittings}]
Suppose that the gluing map $\varphi$ is 
a $k$-th power of the Dehn twist about the binding $b$, where $|k| \geq 5$. 
For the Heegaard splitting $M = H_1 \cup_\Sigma H_2$ as above, we have the following.
\begin{enumerate}
\item
The splitting $M = H_1 \cup_\Sigma H_2$ is not induced from an open book decomposition. 
\item
The Goeritz group of $M = H_1 \cup_\Sigma H_2$ is isomorphic to the mapping class group of $F$. 
In particular, it is an infinite group. 
\end{enumerate} 
\end{citingtheorem}
Note that it follows directly from Yoshizawa \cite{Yos14} 
that the distance of the splitting $M = H_1 \cup_\Sigma H_2$ in the above theorem 
is exactly $2$. 
Theorem~\ref{thm:the Goeritz groups of distance-2 Heegaard splittings} 
indicates that the ``minorities" is not as minor as we wondered in the previous paragraph. 
Further, it is remarkable that the above theorem 
gives the first explicit computation of the infinite-order Goeritz groups of strongly-irreducible Heegaard splittings. 

The second case is, on the contrary, that the gluing map $\varphi$ is complicated in the sense that  
the distance in the curve graph 
$\mathcal{C} (\Sigma_b) $ between the images of the {\it subsurface projection} $\pi_{\Sigma_b}$ of 
the sets of meridian disks $\mathcal{D} (H_1)$ of $H_1$ and $\mathcal{D} (H_2)$ of $H_2$ is sufficiently large, 
where $\Sigma_b := \Cl (\Sigma - \Nbd (b))$. 
In this case, we can show that the distance of the splitting $M = H_1 \cup_\Sigma H_2$ is exactly $4$ and 
we can compute the Goeritz group as follows, where 
the definition of the group $G (S, \iota_0, \iota_1)$ is given in Section~\ref{subsec:The Goeritz groups of distance-4 Heegaard splittings}:  

\begin{citingtheorem}[Theorem~\ref{thm:the Goeritz groups of distance-4 Heegaard splittings}]
Suppose that the distance in $\mathcal{C}(\Sigma_b)$ between 
$\pi_{\Sigma_b} (\mathcal{D} (H_1))$ and $\pi_{\Sigma_b} (\mathcal{D} (H_2))$ is greater than $10$. 
For the Heegaard splitting $M = H_1 \cup_\Sigma H_2$ as above, we have the following.
\begin{enumerate}
\item
The distance of the splitting $M = H_1 \cup_\Sigma H_2$ is exactly $4$. 
\item
The Goeritz group of $M = H_1 \cup_\Sigma H_2$ is isomorphic to the group $G (S, \iota_0, \iota_1)$. 
\end{enumerate} 
\end{citingtheorem}
We will see that there actually exist (generiacally, in some sense) the Heegaard splittings satisfying the condition in Theorem~\ref{thm:the Goeritz groups of distance-4 Heegaard splittings}. 
The Goeritz group in 
Theorem~\ref{thm:the Goeritz groups of distance-4 Heegaard splittings} 
is of course a finite group since the distance of the Heegaard splitting is $4$. 
The existence of a Heegaard splitting of distance $3$ having the infinite-order Goeritz group  still remains open.  

\vspace{1em}

Throughout the paper, 
any curves on a surface, or surfaces in a 3-manifold are always assumed to be properly embedded, 
and their intersection is transverse and minimal up to isotopy. 
For convenience, we usually will not distinguish 
curves, surfaces, maps, etc.  
from their isotopy classes in their notation. 
Let $Y$ be a subspace of a space $X$. 
In this paper, $\Nbd(Y; X)$, or simply $\Nbd(Y)$, will denote a regular neighborhood of $Y$ in $X$, 
$\Cl(Y)$ the closure of $Y$,  and 
$\Int (Y)$ the interior of $Y$ in $X$. 
The number of components of $Y$ is denoted by $\# Y$.  

\section{Preliminaries}
\label{sec:Preliminaries}

\subsection{Curve graphs}
\label{subsec:Curve complexes}

Let $\Sigma$ be a compact surface. 
A simple closed curve on $\Sigma$ is said to be {\it essential} 
if it is not homotopic to a point or a loop around a boundary component of $\Sigma$. 
An arc on $\Sigma$ is said to be {\it essential} 
if it is not homotopic (rel. endpoints) to a subarc of a boundary component of $\Sigma$.

Let $\Sigma$ be a compact orientable surface of genus $g$ with $p$ boundary components. 
We say that $\Sigma$ is {\it sporadic} if $3g + p \leq 4$. 
Otherwise, $\Sigma$ is said to be {\it non-sporadic}. 
Suppose that $\Sigma$ is non-sporadic. 
The {\it curve graph} $\mathcal{C}(\Sigma)$ of $\Sigma$
is the $1$-dimensional simplicial complex whose 
vertices are the isotopy classes of essential simple closed curves on $\Sigma$ 
such that a  pair of distinct vertices spans  an edge 
if and only if they admit disjoint representatives.
Similarly, 
the {\it arc and curve  graph} $\mathcal{AC}(\Sigma)$ of $\Sigma$
is defined to be the  $1$-dimensional simplicial complex 
whose vertices are the isotopy classes of essential 
arcs and simple closed curves on $\Sigma$ 
such that a pair of distinct vertices spans an edge
if  and only if they admit disjoint representatives. 
The sets of vertices of $\mathcal{C}(\Sigma)$ and $\mathcal{AC}(\Sigma)$ 
are denoted by $\mathcal{C}^{(0)}(\Sigma)$ and $\mathcal{AC}^{(0)}(\Sigma)$, respectively. 
We equip the curve graph $\mathcal{C} (\Sigma)$ 
(resp. the arc and curve graph $\mathcal{AC} (\Sigma)$) 
with the simplicial distance $d_{\mathcal{C} (\Sigma)}$ 
(resp. $d_{ \mathcal{AC}} (\Sigma)$). 
Note that both $\mathcal{C} (\Sigma)$ and 
$\mathcal{AC} (\Sigma)$ are geodesic metric spaces. 

Let $Y$ be an {\it essential} (i.e., $\pi_1$-injective), non-sporadic subsurface of $\Sigma$. 
The {\it subsurface projection} $\pi_Y : \mathcal{C}^{(0)}(\Sigma) \to P( \mathcal{C}^{(0)}(Y))$, 
where $P( \cdot )$ denotes the power set, is defined as follows. 
First, define $\kappa_Y : \mathcal{C}^{(0)}(\Sigma) \to P( \mathcal{AC}^{(0)}(Y))$ to be the map that takes 
$\alpha \in \mathcal{C}^{(0)}(S)$ to $\alpha \cap Y$. 
Further, define the map $\sigma_Y : \mathcal{AC}^{(0)}(Y) \to P( \mathcal{C}^{(0)}(Y) )$ by 
taking $\alpha \in \mathcal{AC}^{(0)}(Y)$ to the set of simple closed curves 
on $Y$ consisting of the components of the boundary of $\Nbd (\alpha \cup \partial Y; Y)$ that are essential in $Y$. 
The map $\sigma_Y$ naturally extends to 
a map $\sigma_Y : P( \mathcal{AC}^{(0)}(Y)) \to P( \mathcal{C}^{(0)}(Y) )$. 
The subsurface projection $\pi_Y : \mathcal{C}^{(0)}(\Sigma) \to P( \mathcal{C}^{(0)}(Y))$ is then defined 
by $\pi_Y = \sigma_Y \circ \kappa_Y$. 
See for example Masur-Minsky \cite{MM00} and Masur-Schleimer \cite{MS13} for details. 
The following lemma is straightforward from the definition. 
\begin{lemma}
\label{lem:subsurface projection of a geodesic segment}
Let $( \alpha_0, \ldots, \alpha_n )$ be a geodesic segment in $\mathcal{C}(\Sigma)$. 
If $\alpha_j \cap Y \neq \emptyset$ for each $j \in \{ 0 , \ldots , n \}$, then it holds   
$d_{\mathcal{C}(Y)} (\pi_Y (\alpha_0), \pi_Y (\alpha_n)) \leq 2n$. 
\end{lemma}

\subsection{Distance of a Heegaard splitting}
\label{subsec:Distance of a Heegaard splitting} 


Let $H$ be a handlebody of genus at lest $2$. 
We denote by $\mathcal{D}(H)$ the subset of $\mathcal{C}^{(0)}(\partial H)$ consisting of 
simple closed curves that bound disks in $H$. 
Given a Heegaard splitting $M = H_1 \cup_{\Sigma} H_2$, the {\it distance} $d (H_1, H_2)$ of 
the  splitting is defined by $d (H_1, H_2) = 
d_{ \mathcal{C}(\Sigma)} (\mathcal{D}(H_1), \mathcal{D}(H_2))$. 
We say that a Heegaard splitting $M = H_1 \cup_{\Sigma} H_2$ is {\it strongly irreducible} 
if $d (H_1, H_2) \geq 2$. 

A Heegaard splitting $M = H_1 \cup_{\Sigma} H_2$ 
is said to be {\it keen} if there exists a unique pair of 
$\alpha \in \mathcal{D}(H_1)$ and $\alpha' \in \mathcal{D}(H_2)$ satisfying 
$d_{\mathcal{C}(\Sigma)} ( \alpha , \alpha' ) = d(H_1, H_2)$. 
In particular, $M = H_1 \cup_{\Sigma} H_2$ is said to be {\it strongly keen} if 
there exists a unique geodesic segment $(\alpha = \alpha_0, \alpha_1, \ldots , \alpha_{n-1}, \alpha_n = \alpha')$, where $n = d (H_1, H_2)$, 
such that $\alpha \in \mathcal{D}(H_1)$ and $\alpha' \in \mathcal{D} (H_2)$. 
We say that a Heegaard splitting $M = H_1 \cup_{\Sigma} H_2$ is {\it weakly keen} 
if there exist only finitely many pairs of 
$\alpha \in \mathcal{D}(H_1)$ and $\alpha' \in \mathcal{D}(H_2)$ satisfying 
$d_{\mathcal{C}(\Sigma)} ( \alpha , \alpha' ) = d(H_1, H_2)$. 
The notion of a keen (and a strongly keen) Heegaard splitting was first introduced by 
Ido-Jang-Kobayashi \cite{IJK18}, 
who showed the following theorem.  
\begin{theorem}[Ido-Jang-Kobayashi \cite{IJK18}]
\label{thm:IJK18}
For any $g \geq 3$ and $n \geq 2$, there exists a genus-$g$ strongly keen Heegaard splitting 
$M = H_1 \cup_{\Sigma} H_2$ with $d (H_1, H_2) = n$. 
\end{theorem}

\subsection{Mapping class groups}
\label{subsec:Mapping class groups}

Let $Y_1, \ldots, Y_n$ be possibly empty subspaces 
of a compact manifold $X$. 
We denote by $\Aut (X, Y_1, \ldots, Y_n)$ 
the group of automorphisms of $X$ which {map} $Y_i$ onto 
$Y_i$ for any $i=1 , \ldots , n$. 
The {\it mapping class group} of $(X , Y_1 , \ldots, Y_n )$, denoted by 
$\MCG (X , Y_1 , \ldots, Y_n )$, 
is defined to be the group of connected components  of $\Aut (X , Y_1 , \ldots, Y_n)$. 
The equivalence class in $\MCG (X , Y_1 , \ldots, Y_n )$ of a map in 
$\Aut (X , Y_1 , \ldots, Y_n )$ is called its {\it mapping class}. 
As mentioned in the introduction, we usually will not distinguish 
a map and its mapping class. 
This should not cause any confusion since it will usually be clear from the context 
in which equivalence relation we consider for the maps in question. 
When $X$ is orientable, 
the ``plus" subscripts, for instance in 
$\Aut_+ (X, Y_1, \ldots, Y_n )$ and 
$\MCG_+ (X, Y_1, \ldots, Y_n)$, 
indicate the subgroups of 
$\Aut (X, Y_1,  \ldots, Y_n  )$ and 
$\MCG (X, Y_1, \ldots, Y_n )$, 
respectively, consisting of 
orientation-preserving automorphisms 
(or their mapping classes) of $X$. 	 

Let $M = H_1 \cup_{\Sigma} H_2$ be a Heegaard splitting. 
The group $\MCG_+ (M, H_1)$ is called 
the {\it mapping class group}, or the {\it Goeritz group}, of the splitting. 
Note that the natural map  
$\MCG_+ (M, H_1) \to \MCG_+ (\Sigma)$ that 
takes (the mapping class of) $\varphi \in \MCG_+ (M, H_1)$ to (that of) 
$\varphi|_{\Sigma} \in \MCG_+(\Sigma)$ 
is injective. 
In this way $\MCG_+ (M, H_1)$  can be naturally regarded as a subgroup of $\MCG_+ (\Sigma)$. 
In general, an automorphism $\psi$ of a submanifold $Y$ of a manifold $X$ is 
said to be {\it extendable over $X$} if  
$\psi$ extends to an automorphism of the pair $(X, Y)$. 
We can say that the Goeritz group for the splitting $M = H_1 \cup_\Sigma H_2$ 
is the subgroup of the mapping class group 
$\MCG_+(\Sigma)$ of the Heegaard surface $\Sigma$ consisting of 
elements that are extendable over $M$. 


In this paper, the mapping class groups of non-orientable surfaces 
will also be particularly important. 
Let $F$ be a compact non-orientable surface with non-empty boundary. 
Let $p : \Sigma \to F$ be the orientation double-cover. 
Since the set of two-sided loops are preserved by any automorphism of $F$, 
any map $\varphi \in \Aut (F)$ lifts to 
a unique orientation-preserving automorphism of $\Sigma$. 
(The other lift of $\varphi$ is orientation-reversing.)  
This gives a well-defined homomorphism $L : \MCG (F) \to \MCG_+ (\Sigma)$. 
We use the following easy but important lemma in 
Section~\ref{subsec:The Goeritz groups of distance-2 Heegaard splittings}. 

\begin{lemma}
\label{lem:injectivity of the map from MCG(F) to MCG+(Sigma)}
The above map $L : \MCG (F) \to \MCG_+ (\Sigma)$ is injective. 
\end{lemma}
\begin{proof}
Let $F \tilde\times I$ be the orientable twisted product, which is a handlebody, and 
$\pi: F \tilde\times I \to F$ the natural projection. 
We identify $\Sigma$ with $F \tilde\times \partial{I} \subset F \tilde\times I$, and 
$F$ with $F \tilde\times \{ 1/2 \} \subset F \tilde\times I$. 
Note that $\pi|_\Sigma$ is nothing but the orientation double cover $p : \Sigma \to F$. 

Let $\varphi_F$ be an automorphism of $F$ whose mapping class belongs to 
the kernel of $L: \MCG (F) \to \MCG_+ (\Sigma)$.  
The map $\varphi_F$ extends to a fiber-preserving homeomorphism $\Phi \in \Aut_+ (F \tilde\times I)$ 
with $\varphi := \Phi|_{\Sigma} = L(\varphi_F)$.  
The map $\varphi$ is isotopic to the identity $\id_{\Sigma}$, thus, 
$\Phi|_{\partial (F \tilde\times I )}$ can be described as 
\[ \Phi|_{\partial (F \tilde\times I )} = \tau_{c_1}^{k_1} \circ \cdots \circ \tau_{c_n}^{k_n} , \]
where $c_1, \ldots, c_n$ are the connected components of $\partial F$, 
$\tau_{c_i}^{k_i}$ is the Dehn twist about the simple closed curve $c_i$ ($i=1, \ldots , n$), 
and  $k_1, \ldots , k_n$ are integers.  
Since each $c_i$ does not bound a disk in $F \tilde\times I$, 
and each pair of $c_i$ and $c_j$ ($1 \leq i < j \leq n$) does not cobound an annulus, 
we have $k_1 = \cdots = k_n = 0$
due to Oertel \cite{Oer02} or McCullough \cite{McC06}. 
Therefore,  $\Phi$ is isotopic to the identity $\id_{\partial (F \tilde\times I )}$, 
so  $\Phi$ is isotopic to the identity $\id_{F \tilde\times I }$. 
Since the inclusion $\iota : F \to F \tilde\times I$ is a homotopy equivalence with $\pi$ a homotopy inverse, 
the composition $\pi \circ \Phi \circ \iota$ is homotopic to $\id_F$. 
It follows that $\varphi_F$ is homotopic to the identity. 
Now by Epstein \cite{Eps66}, $\varphi_F$ is isotopic to $\id_F$. 
\end{proof}

\subsection{Pants decompositions and twisting numbers}
\label{subsec:Pants decompositions and twisting numbers}

Let $\Sigma$ be a closed orientable surface of genus $g$, where $g \geq 2$. 
The set of $3g - 3$ mutually disjoint, mutually non-isotopic, essential simple closed curves 
on $\Sigma$ is called a {\it pants decomposition} of $\Sigma$. 
Let $\mathscr{P}$ be a pants decomposition of $\Sigma$. 
Let $C$ be the union of the simple closed curves of $\mathscr{P}$. 
Let $\alpha$ be an essential arc on a component $P$, which is a pair of pants, of 
$\Cl (\Sigma - \Nbd (C))$. 
We call $\alpha$ a {\it wave for $\mathscr{P}$}  
if the both endpoints of $\alpha$ lie on the same component of $\partial P$. 
Otherwise, $\alpha$ is called a {\it seam for $\mathscr{P}$}. 
Let $k > 0$. 
An essential simple closed curve $\beta$ on $\Sigma$ (that intersects $C$ minimally up to isotopy) 
is said to be 
{\it $k$-seamed with respect to $\mathscr{P}$} if for each component $P$ 
of $\Cl (\Sigma - \Nbd (C))$, 
there exist at least $k$ arcs of $\beta \cap P$ representing each of the three distinct isotopy classes of 
seams for $\mathscr{P}$. 

Let $l$ be a simple closed curve on a closed oriented surface $\Sigma$ of genus at least $2$. 
We denote by $\tau_l$ the (left-handed) {\it Dehn twist} about $l$. 
Let $\mathscr{P}$ be a pants decomposition of $\Sigma$. 
Let $C$ be the union of the simple closed curves of $\mathscr{P}$. 
Set $N := \Nbd (l)$. 
Fix an identification of $N$ with the product $l \times I$, where 
$l$ corresponds to $l \times \{ 1/2 \}$. 
We may assume that each component of $N \cap C$ is an $I$-fiber of $N$. 
Let $\alpha$ be an essential simple arc on $N$ with the endpoints disjoint from $N \cap C$ that intersects 
each $I$-fiber of $N$ transversely. 
Then the {\it twisting number} of $\alpha$ in $N$ with respect to $C$ is defined as follows. 
Let $p$ be an endpoint of $\alpha$. 
Let $v_\alpha$ be the inward-pointing tangent vector of $\alpha$ based at $p$. 
Likewise, let $v_I$ be the inward-pointing tangent vector based at $p$ of the $I$-fiber of $N$ with $p$ an endpoint. 
If the pair $(v_\alpha, v_I)$ is compatible with the orientation of $\Sigma$, 
the twisting number is defined to be $ \# (\alpha \cap C) / \# (N \cap C) \in \Rational$. 
Otherwise, it is defined to be $ - \# (\alpha \cap C) / \# (N \cap C) \in \Rational$.  
See Figure~\ref{fig:twisting_number}. 
\begin{figure}[htbp]
\begin{center}
\includegraphics[width=5cm,clip]{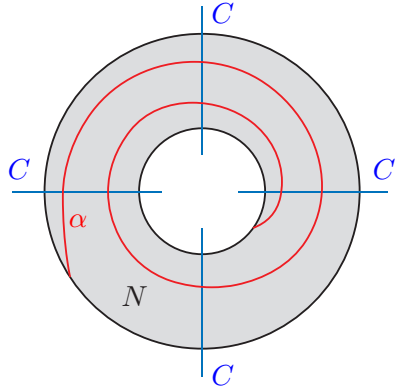}
\begin{picture}(400,0)(0,0)
\put(150,70){\color{red}$\alpha$}
\put(127,88){\color{blue}$C$}
\put(265,88){\color{blue}$C$}
\put(204,10){\color{blue}$C$}
\put(204,147){\color{blue}$C$}
\put(170,40){$N$}
\end{picture}
\caption{The twisting number of the arc $\alpha$ in $N$ with respect to $C$ is $7 / 4$.}
\label{fig:twisting_number}
\end{center}
\end{figure}
We refer the reader to Yoshizawa \cite{Yos14} for more details on the twisting numbers. 

Let $\Sigma$, $l$, $N$, $\mathscr{P}$ and $C$ be as above. 
Let $\beta$ be a simple closed curve on $\Sigma$. 
We say that $\beta$ is in {\it efficient position} with respect to $(N, C)$ if 
\begin{itemize}
\item
$\beta$ intersects $\partial N$ and $C$ minimally (up to isotopy); 
\item
$\beta$ intersects each $I$-fiber of $N$ transversely; and 
\item
$\beta \cap C \cap \partial N = \emptyset$. 
\end{itemize}
Suppose that $\beta$ is in efficient position with respect to 
$(N, C)$.  
A disk $E$ in $\Sigma - \Int (N)$ is called an {\it outer triangle} of $N$ with respect to 
$(N, \mathscr{P}, \beta)$ if 
$\partial E \subset \partial N \cup C \cup \beta$ and 
each of $\partial E \cap \partial N$, $\partial E \cap C$, $\partial E \cap \beta$  
is a single arc. See Figure~\ref{fig:outer_triangle}.  
\begin{figure}[htbp]
\begin{center}
\includegraphics[width=5cm,clip]{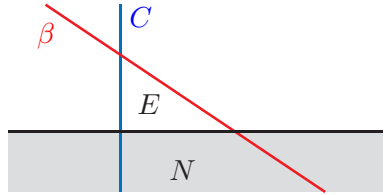}
\begin{picture}(400,0)(0,0)
\put(140,70){\color{red}$\beta$}
\put(175,76){\color{blue}$C$}
\put(190,18){$N$}
\put(178,43){$E$}
\end{picture}
\caption{An outer trianble $E$ of $N$ with respect to $(N, \mathscr{P}, \beta)$.}
\label{fig:outer_triangle}
\end{center}
\end{figure}
Note that we can perform an isotopy of $\beta$ keeping that $\beta$ 
is in efficient position with respect to $(N, C)$ so that 
$(N, \mathscr{P}, \beta)$ admits no outer triangles. 

\begin{lemma}[Yoshizawa \cite{Yos14}]
\label{lem:minimality of twisting number and outer triangles}
Let $\Sigma$, $l$, $N$, $\mathscr{P}$ and $C$ be as above. 
Let $\beta$ be a simple closed curve on $\Sigma$ in efficient position with respect to $(N, C)$ 
such that $(N, \mathscr{P}, \beta)$ admits no outer triangles. 
Let $\alpha_1, \ldots, \alpha_r$ be the components of $\beta \cap N$, and 
$t_j$ $(j \in \{1 , \ldots , r \})$ the twisting number of $\alpha_j$ in $N$ with respect to $C$. 
Let $k$ be an integer such that either $ k + t_j \geq 0$ $($for all $j)$ or 
$ k + t_j \leq 0$ $($for all $j)$. 
Then  $\tau_l^k ( \beta )$ remains to be in efficient position with respect to $(N , C)$, 
and the twisting number of $\tau_l^k (\alpha_j)$ in $N$ with respect to $C$ is $k + t_j$. 
\end{lemma}

The following lemma, which we will use in 
Section \ref{subsec:The Goeritz groups of distance-2 Heegaard splittings}, is straightforward from the definitions. 
\begin{lemma}
\label{lem:lower bound for the number of seams}
Let $\Sigma$, $l$, $N$, $\mathscr{P}$, $\beta$, $\alpha_j$ and $t_j$
$(j \in \{ 1 , \ldots , r \})$ be as in Lemma~$\ref{lem:minimality of twisting number and outer triangles}$. 
If $l$ is $1$-seamed with respect to the pants decomposition $\mathscr{P}$ and 
there exists $j$ with $|t_j| > k$, then $\beta$ is $k$-seamed with respect to $\mathscr{P}$. 
\end{lemma}

\subsection{Measured laminations}
\label{subsec:Measured laminations}

In this subsection, $\Sigma$ denotes a compact (possible non-orientable) surface with 
$\chi(\Sigma) < 0$, where $\chi ( \cdot )$ denotes the Euler characteristic. 
We fix a hyperbolic metric on $\Int \, \Sigma$. 
The main references for this subsection are Fathi-Laudenbach-Po\'{e}naru \cite{FLP79} and 
Penner-Harer \cite{PH92}. 

Recall that a {\it geodesic lamination} on $\Sigma$ is a foliation of a non-empty closed subset of $\Sigma$ by geodesics. 
A {\it transverse measure} $m$ for a geodesic lamination $\lambda$ is a function that 
assigns a positive real number to each smooth compact arc transverse to $\lambda$ 
so that $m$ is invariant under isotopy respecting the leaves of $\lambda$. 
A geodesic lamination equipped with a transverse measure is called a {\it measured geodesic lamination}. 
The set $\ML(\Sigma)$ of measured geodesic laminations on $\Sigma$  can be equipped with 
the weak-* topology, 
for which two measured geodesic laminations are close if they induce approximately the same measures 
on any finitely many arcs transverse to them. 
The quotient $\mathbb{P} \ML (\Sigma)$ of $\ML (\Sigma) $ under 
the natural action of the multiplicative group $\Real_+ := (0, \infty)$ is called the 
{\it projective measured geodesic lamination}. 

\begin{theorem}[Thurston \cite{Thu88}]
\label{thm:PL structure on ML}
\begin{enumerate}
\item
The space $\ML (\Sigma)$ $($resp. $\mathbb{P} \ML (\Sigma)$$)$ 
admits a natural piecewise linear $($resp. piecewise projective$)$ structure. 
\item
There exists a piecewise linear $($resp. piecewise projective$)$ homeomorphism between 
$\ML (\Sigma)$ $($resp. $\mathbb{P} \ML (\Sigma)$$)$ and 
$\Real^{6g + 3 h + 2n - 6} - \{ 0 \}$  
$($resp. $S^{6g + 3h + 2n - 7}$$)$, 
where $\Sigma \cong ( \#_{g} T^2 ) \# (\#_h \mathbb{RP}^2 ) - \sqcup_{n} \Int D^2 $. 
\end{enumerate}
\end{theorem}

A multiset of pairwise disjoint, pairwise non-isotopic, closed geodesics on $\Sigma$ 
is called a {\it weighted multicurve}. 
The set of multicurves on $\Sigma$ is denoted by $\mathcal{S} (\Sigma)$. 
Using the Dirac mass, we regard 
$\mathcal{S} (\Sigma)$ as a subset of $\mathbb{P} \ML(\Sigma)$. 
We will use the following theorem in 
Section~\ref{subsec:The Goeritz groups of distance-4 Heegaard splittings}. 

\begin{theorem}[see Penner-Harer \cite{PH92}]
\label{thm:multicurves in ML}
The set $\mathcal{S} (\Sigma)$ is dense in $\mathbb{P} \ML(\Sigma)$. 
\end{theorem}

We regard closed geodesics on $\Sigma$ as points in $\ML (\Sigma)$. 
For simple closed geodesics $\alpha$ and $\beta$ on $\Sigma$, 
$i (\alpha , \beta )$ denotes the geometric intersection number. 
For $(\lambda, m) \in \ML (\Sigma)$, $i ( (\lambda, m), \alpha ) $ is 
defined to be the minimal transverse length with respect to the measure $m$ for $\lambda$. 
\begin{theorem}[Rees \cite{Ree81}]
\label{thm:intersection numbers of measured geodesic laminations}
The above $i ( \cdot, \cdot)$ extends to a continuous function 
$\ML (\Sigma) \times \ML (\Sigma) \to \Real$ 
that is bilinear and invariant under the action of $\MCG (\Sigma)$. 
\end{theorem}

\section{The Goeritz groups of keen Heegaard splittings}
\label{sec:The Goeritz groups of keen Heegaard splittings}

In this section, we discuss the finiteness of the Goeritz groups 
of keen Heegaard splittings. 

\begin{proposition}
\label{thm:The Goeritz groups of keen Heegaard splittings}
Let $M = H_1 \cup_\Sigma H_2$ be a Heegaard splitting of genus at least $2$. 
\begin{enumerate}
\item
If  the splitting $M = H_1 \cup_\Sigma H_2$ is strongly keen and the distance $d(H_1, H_2)$ is $2$, the Goeritz group 
$\MCG_+ (M, H_1)$ is either a finite cyclic group or a finite dihedral group. 
\item
If  the splitting $M = H_1 \cup_\Sigma H_2$ is keen and the distance $d(H_1, H_2)$ is at least $3$, the Goeritz group 
$\MCG_+ (M, H_1)$ is either a finite cyclic group or a finite dihedral group. 
\item
If the splitting $M = H_1 \cup_\Sigma H_2$ is weakly keen and the distance 
$d (H_1, H_2)$ is at least $3$, the Goeritz group 
$\MCG_+ (M, H_1)$ is a finite group. 
\end{enumerate}
\end{proposition}
\begin{proof}
\noindent (1) 
Suppose that  $M = H_1 \cup_\Sigma H_2$ is strongly keen and $d(H_1, H_2) = 2$. 
There exists a unique geodesic segment $(\alpha_0, \alpha_1, \alpha_2)$ 
such that $\alpha_0 \in \mathcal{D} (H_1)$ and $\alpha_2 \in \mathcal{D} (H_2)$. 
Let $\varphi \in \MCG_+ (M, H_1)$. 
By the uniqueness of the geodesic segment $(\alpha_0, \alpha_1, \alpha_2)$, 
we have $\varphi (\alpha_j) = \alpha_j$ for each $j \in \{ 0, 1 , 2 \}$. 
Thus the group $\MCG_+(M, H_1)$ acts on the pair $(\alpha_0, \alpha_0 \cap \alpha_2)$ in a natural way. 
It suffices to show that the action of $\MCG_+(M, H_1)$ on $(\alpha_0, \alpha_0 \cap \alpha_2)$ is faithful, which in turn implies 
that $\MCG_+(M, H_1)$ is either a finite cyclic group or a finite dihedral group. 
Let $\psi$ be an element of $\MCG_+ (M, H_1)$ that acts on 
$(\alpha_0, \alpha_0 \cap \alpha_2)$ trivially. 
Since $\psi$ is orientation-preserving, $\psi$ preserves an orientation of 
$\alpha_2$. 
Therefore, we can assume that 
$\psi | _{\alpha_0 \cup \alpha_2}$ 
is the identity on $\alpha_0 \cup \alpha_2$.
Since the Heegaard splitting $M = H_1 \cup_\Sigma H_2$ is strongly keen and $d (H_1, H_2) = 2$, 
$\Cl (\Sigma - \Nbd (\alpha_0 \cup \alpha_2))$ consists of 
finitely many disks and a single annulus, and $\alpha_1$ is the core of that  annulus. 
By the Alexander trick, we can assume that $\psi$ is the identity outside of 
the annulus $\Nbd (\alpha_1)$. 
Thus $\psi$ is a power $\tau_{\alpha_1}^n$ of the Dehn twist $\tau_{\alpha_1}$. 
If $n \neq 0$, the circle $\alpha_1$ bounds disks both in $H_1$ and $H_2$ 
due to Oertel \cite{Oer02} or McCullough \cite{McC06}, which is a contradiction. 
Therefore, $\psi$ is the identity.

\noindent (2)
Suppose that $M = H_1 \cup_\Sigma H_2$ is keen and $d(H_1, H_2) \geq 3$. 
This case is easier than (1). 
Since $M = H_1 \cup_\Sigma H_2$ is keen, 
there exists a unique pair of $\alpha \in \mathcal{D} (H_1)$ and 
$\alpha' \in \mathcal{D} (H_2)$ satisfying  
$d_{\mathcal{C} (\Sigma)} (\alpha, \alpha') = d (H_1, H_2)$. 
Thus any $\varphi \in \MCG_+ (M , H_1)$ preserves 
both $\alpha$ and $\alpha'$. 
Since $d (H_1, H_2) \geq 3$, $\Cl (\Sigma \setminus \Nbd (\alpha \cup \alpha'))$ 
consists only of disks. 
Thus, the same argument as in the proof of (1) shows that 
$\MCG_+ (M, H_1)$ is either a finite cyclic group or a finite dihedral group.  

\noindent (3)
Suppose that $M = H_1 \cup_\Sigma H_2$ is weakly keen and $d(H_1, H_2) \geq 3$.  
In this case,  
we can show that the order of any $\varphi \in \MCG_+ (M, H_1)$ is finite as follows.  
Let $\varphi \in \MCG_+ (M, H_1)$. 
Choose $\alpha \in \mathcal{D} (H_1)$ and $\alpha' \in \mathcal{D} (H_2)$ 
such that $d_{\mathcal{C} (\Sigma)} (\alpha, \alpha') = d (H_1, H_2)$. 
Since the Heegaard splitting $M = H_1 \cup_\Sigma H_2$ is weakly keen, 
there exists an integer $n$ such that $\varphi^n (\alpha) = \alpha$ and 
$\varphi^n (\alpha') = \alpha'$. 
Since $d (H_1, H_2) \geq 3$, 
$\Cl (\Sigma - \Nbd (\alpha \cup \alpha'))$ consists only of disks. 
Thus, the same argument 
as above shows that the order of the element $\varphi^n$ is finite in 
$\MCG_+ (M , H_1)$. 
Due to Serre \cite{Ser60}, any torsion subgroup of $\MCG_+ (\Sigma)$ is a finite group. 
The above argument therefore immediately 
implies that the Goeritz group $\MCG_+ (M, H_1)$ is a finite group. 
\end{proof}

As a direct corollary of Proposition~\ref{thm:The Goeritz groups of keen Heegaard splittings} and 
Theorem \ref{thm:IJK18}, we get the following: 

\begin{corollary}
\label{cor:direct corollary of IJK18}
For any $g \geq 3$ and $n \geq 2$, there exists a genus-$g$ Heegaard splitting 
$M = H_1 \cup_{\Sigma} H_2$ with $d (H_1, H_2) = n$ such that 
the Goeritz group $\MCG_+ (M, H_1)$ is either a finite cyclic group or a finite dihedral group. 
\end{corollary}

\section{Handlebodies as interval bundles}
\label{sec:Handlebodies as interval bundles}

Let $F$ be a compact (possibly non-orientable) surface with non-empty boundary. 
Let $\pi: H \to F$ be the orientable $I$-bundle. 
Note that $H$ is a handlebody and $\pi^{-1}(\partial F)$ consists of annuli on $\partial H$. 
We call the union of the core curves of $\pi^{-1}(\partial F)$ the {\it binding} of this $I$-bundle. 
In this paper, we often identify $F$ with the image $F \times \{ 1/2 \}$ of a section 
of the $I$-bundle $H \to F$, and under this identification, we regard that 
$b = \partial F$. 
The union of disjoint simple closed curves on the boundary $\partial H$ of a handlebody $H$ 
is called a {\it binding} of $H$ if it is the binding of an $I$-bundle structure $H \to F$. 

In the following, let $H$ be a handlebody of genus $g$, where $g \geq 2$. 

\begin{lemma}
\label{lem:distance of the binding}
If a simple closed curve $b$ on $\partial H$ is a binding,  
we have $d_{\mathcal{C} (\Sigma)} (b , \mathcal{D} (H)) = 2$. 
\end{lemma}
\begin{proof}
Since $b$ is connected and $\partial H - b$ is incompressible in $H$, 
the distance $d_{\mathcal{C}(\Sigma)} (b, \mathcal{D} (H))$ is 
at least $2$. 
Let $\pi : H \to F$ be the $I$-bundle such that $b$ is its binding. 
Let $\alpha$ be an essential arc on $F$. 
Then $D := \pi^{-1} (\alpha)$ is an essential disk in $H$. 
Since the Euler characteristic of $F$ is negative, there exists a null-homotopic simple closed curve 
$\beta$ on $F$ disjoint from $\alpha$. 
Then $A := \pi^{-1} (\beta)$ is an annulus or a M\"{o}bius band in $H$ that satisfies 
$\partial D \cap \partial A = \emptyset$ and $\partial A \cap \partial b = \emptyset$. 
Thus we have $d_{\mathcal{C}(\Sigma)} (b, \mathcal{D} (H)) = 2$. 
See Figure~\ref{fig:binding}. 
\begin{figure}[htbp]
\begin{center}
\includegraphics[width=6.5cm,clip]{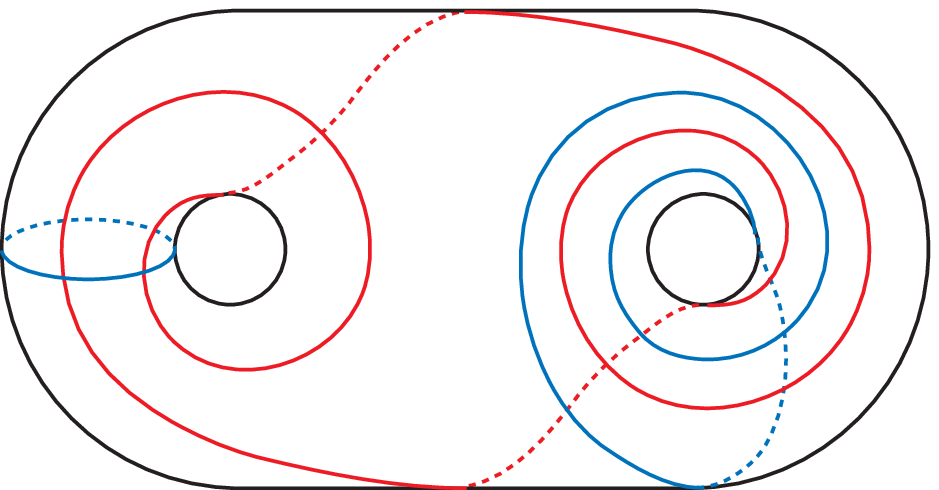}
\begin{picture}(400,0)(0,0)
\put(185,56){\color{red}$b$}
\put(88,56){\color{blue}$\partial D$}
\end{picture}
\caption{This figure depicts the case where the genus of $H$ is two and 
$F$ is non-orientable. 
The distance between $\partial D$ and $b$ in $\mathcal{C} (\Sigma)$ is two.}
\label{fig:binding}
\end{center}
\end{figure}
\end{proof}

The set of $3g - 3$ mutually disjoint, mutually non-isotopic, essential disks in $H$ is called 
a {\it solid pants decomposition} of $H$. 
Let $\mathscr{S} = \{ D_1, \ldots, D_{3g-3} \}$ be a solid pants decomposition of $H$. 
An essential arc $\alpha$ on a component $P$ of 
$\Cl (\partial H - \Nbd ( \bigcup_{i=1}^{3g-3} \partial D_i))$ is called a {\it wave} (resp. {\it seam}) 
{\it for $\mathscr{S}$} if it is a wave (resp. seam) for the pants decomposition 
$\mathscr{P} = \{ \partial D_1, \ldots, \partial D_{3g-3} \}$ of the surface $\partial H$. 
An essential simple closed curve $\beta$ on $\partial H$ is said to be {\it $k$-seamed with respect to 
$\mathscr{S}$} if $\beta$ is $k$-seamed with respect to the pants decomposition $\mathscr{P}$ of $\partial H$. 

The proof of the following lemma is straightforward. 

\begin{lemma}
\label{lem:solid pants decomposition and waves}
Let $\mathscr{S}$ be a solid pants decomposition of $H$. 
Then the boundary of each essential disk $D$ in $H$ with $D \notin \mathscr{S}$ 
contains at least two waves for $\mathscr{S}$. 
\end{lemma}

\begin{lemma}
\label{lem:the binding is 1-seamed}
Each binding $b$ of $H$ admits 
a solid pants decomposition $\mathscr{S}$ of $H$ such that 
$b$ is $1$-seamed with respect to $\mathscr{S}$. 
\end{lemma}
\begin{proof}
Let $\pi : H \to F$ be the $I$-bundle such that $b$ is its binding. 
Let $\{ \alpha_1 , \ldots , \alpha_n \}$ be a maximal collection of 
mutually disjoint, mutually non-isotopic, essential arcs on $F$. 
Then $\{ \pi^{-1}(\alpha_1), \ldots , \pi^{-1} (\alpha_n) \}$ forms the 
required solid pants decomposition of $H$. 
\end{proof}

\begin{lemma}
\label{lem:obstruction to be a binding}
Let $\beta$ be an essential simple closed curve on $\partial H$. 
If $\beta$ is $2$-seamed with respect to a solid pants decomposition 
$\mathscr{S}$ of $H$, then 
$\beta$ is not a binding of $H$. 
\end{lemma}
\begin{proof}
Suppose that $\beta$ is $2$-seamed with respect to a solid pants decomposition 
$\mathscr{S}$ of $H$. 
Let $D$ be an essential disk in $H$. 
If $D$ is a member of $\mathscr{S}$, we have 
\[ i (\beta, \partial D) \geq 4 ,   \]
where $i(\cdot , \cdot)$ is the geometric intersection number. 
Otherwise, by Lemma~\ref{lem:solid pants decomposition and waves}, 
$\partial D$ contains at least two waves $\alpha_1$, $\alpha_2$ with respect to 
$\mathscr{S}$. 
Thus, in this case, we have 
\[ i (\beta, \partial D) \geq \# (\beta \cap \alpha_1) + \# (\beta \cap \alpha_2) \geq 4. \]
See Figure~\ref{fig:pants}. 
\begin{figure}[htbp]
\begin{center}
\includegraphics[width=4cm,clip]{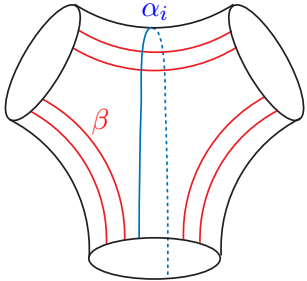}
\begin{picture}(400,0)(0,0)
\put(176,70){\color{red}$\beta$}
\put(195,113){\color{blue}$\alpha_i$}
\end{picture}
\caption{Each wave $\alpha_i$ intersects $\beta$ in at least two points.}
\label{fig:pants}
\end{center}
\end{figure}
Consequently, for any essential disk $D$ in $H$ we have $i (\beta, \partial D) > 2$.  

On the other hand, it is easily seen that for any binding $b$ of $H$, there exists an essential disk $D$ in $H$ with 
$i (b, \partial D) = 2$. 
This implies that $\beta$ is not a binding of $H$. 
\end{proof}

\section{Open and twisted book decompositions}
\label{sec:Open and  twisted book decompositions}

In this section, we consider two analogous structures on a 
closed orientable 3-manifold, {\it open} and {\it twisted book decompositions}. 
Both decompositions naturally induce Heegaard splittings, 
where each handlebody of the splittings inherits the structure of an $I$-bundle. 

Let $S$ be a compact orientable surface with non-empty boundary. 
Let $h$ be an orientation preserving automorphism of $S$ that fixes $\partial S$. 
Consider the mapping torus $S(h)$, which is 
the result of taking $S \times I$ and gluing $S \times \{1\}$ to $S \times \{0\}$ according to $h$. 
The boundary of $S(h)$ can naturally be identified with $\partial S \times S^1$. 
By shrinking each circle $\{x\} \times S^1$, where $x \in \partial S$, to a point, we obtain a 
closed orientable 3-manifold $M$. 
In this paper we shall call such a pair $(S, h)$ an {\it open book decomposition} of $M$. 
The image $b$ of $\partial S \times I$ under the quotient map $q: S \times I \to M$ 
forms a link in $M$. 
We call $b$ the {\it binding} of the open book decomposition $(S, h)$. 
The images $H_1$ and $H_2$ of $S \times [0, 1/2]$ and $S \times [1/2, 1]$, respectively, 
under the quotient map $q$ give a Heegaard splitting for $M$, that is, 
$H_1$ and $H_2$ are handlebodies in $M$ satisfying $H_1 \cup H_2 = M$ and 
$H_1 \cap H_2 = \partial H_1 = \partial H_2$. 
We call this one {\it the Heegaard splitting of $M$ induced 
from the open book decomposition $(S, h)$}. 
Note that the Heegaard surface of the splitting is homeomorphic to the double of $S$. 

Again, let $S$ be a compact orientable surface with non-empty boundary. 
Let $\iota_0$ and $\iota_1$ be orientation reversing, fixted-point-free involutions of $S$ satisfying   
$\iota_0 |_{\partial S} = \iota_1 |_{\partial S}$. 
Remark that, here, the number of the boundary components of $S$ must be even. 
Let $C_1, \ldots , C_{2n}$ be the boundary components of $\partial S$ such that 
$\iota_0 (C_i) = C_{i+n}$ (subscripts $\pmod n$). 
Consider the resulting space $S(\iota_0, \iota_1)$ 
of taking $S \times I$ and gluing $S \times \{0\}$ to itself according to $\iota_0$ and 
$S \times \{1\}$ to itself according to $\iota_1$. 
 The boundary of $S(\iota_0, \iota_1)$ consists of $n$ copies of the torus. 
 For each point $x$ in $\cup_{i=1}^n C_i$, 
 the image of the union $( \{ x \} \times [0,1] ) \cup ( \{ \iota_0(x) \} \times [0,1] ) $ 
 under the quotient map  $S \times [0,1] \to S (\iota_0 , \iota_1)$ is a circle on the boundary tori.  
By shrinking each such circle to a point, we obtain a 
closed orientable 3-manifold $M$. 
We call such a triple $(S, \iota_0, \iota_1)$ a {\it twisted book decomposition} of $M$. 
The image $b$ of $\partial S \times I$ under the quotient map $q: S \times I \to M$ 
forms a link in $M$. 
We call $b$ the {\it binding} of the twisted book decomposition $(S, \iota_0, \iota_1)$. 
The images $H_1$ and $H_2$ of $S \times [0, 1/2]$ and $S \times [1/2, 1]$, respectively, 
under the quotient map $q$ gives a Heegaard splitting for $M$. 
We call this one {\it the Heegaard splitting of $M$ induced from the twisted book decomposition $(S, \iota_0, \iota_1)$}. 
Since $\Sigma := q (S \times \{ 1/2 \})$ is the Heegaard surface of the splitting, 
the surface $\Sigma_b := \Cl (\Sigma - \Nbd (b)) $ is homeomorphic to $S$. 

Note that if $(S, h)$ is an open book decomposition of $M$ with the binding $b$, 
$\Cl (M - \Nbd (b))$ admits a natural foliation with all leaves (called {\it pages}) homeomorphic to $S$. 
Similarly, if $(S, \iota_0, \iota_1)$ is a twisted book decomposition of $M$ with the binding $b$, 
$\Cl (M - \Nbd (b))$ admits a natural foliation with all but two leaves (called {\it pages}) homeomorphic to $S$, 
where the two exceptional leaves are homeomorphic to the non-orientable surface 
$S / \iota_0$ $(\cong S/ \iota_1)$.

\begin{lemma}
\label{lem:constructing open and twisted books from I-bundles}
Let $F$ be a compact surface with non-empty boundary. 
Let $H \to F$ be the orientable $I$-bundle with the binding $b$. 
Let $H_1$ and $H_2$ be copies of $H$. 
Let $M$ be a closed orientable $3$-manifold obtained by gluing $H_1$ to $H_2$ according to  
an automorphism of $\partial H$ preserving $b$. 
Then we have the following: 
\begin{enumerate}
\item
If $F$ is orientable, the resulting Heegaard splitting $M = H_1 \cup H_2$ is induced from an open book decomposition 
where $b$ is the binding. 
\item
If $F$ is non-orientable, $M = H_1 \cup H_2$ is induced from a twisted book decomposition 
where $b$ is the binding. 
\end{enumerate}
\end{lemma}
\begin{proof}
The first assertion is clear from the definition. 
Suppose that $F$ is non-orientable. 
For each $i \in \{ 1, 2 \}$, let $F_i$ be the surface in $H_i$ corresponding to 
the section $F \times \{ 1 / 2 \}$ of the twisted $I$-bundle $H \to F$. 
Set $\Sigma_b := \Cl (\Sigma - \Nbd (b))$, where $\Sigma$ is the Heegaard surface of the 
splitting $M = H_1 \cup H_2$. 
Then $\Cl (M - \Nbd (F_1 \cup F_2))$ is homeomorphic to $\Sigma_b \times I$, which gives 
the structure of a twisted book decomposition of $M$. 
The assertion is now clear from the construction. 
\end{proof}

\section{The Goeritz groups of the Heegaard splittings induced from twisted book decompositions}
\label{sec:The Goeritz groups of the Heegaard splittings induced from twisted book decompositions}

In this section, let $F$ denote a compact non-orientable surface of negative Euler characteristic 
with a single boundary component,  
let $\pi: H \to F$ denote the orientable $I$-bundle with the binding $b$. 
We set $\Sigma := \partial H$, and 
$M = H_1 \cup_\Sigma H_2$ always denotes the Heegaard splitting, where 
$H_1$ and $H_2$ are copies of $H$, and $M$ is obtained by gluing 
$H_1$ to $H_2$ according to an automorphism $\varphi \in \Aut_+ (\Sigma , b)$. 
Note that by Lemma~\ref{lem:constructing open and twisted books from I-bundles}, 
$M = H_1 \cup_\Sigma H_2$ is induced from a twisted book decomposition where $b$ is the binding. 
By Lemma 3.1, the distance $d (H_1 , H_2)$ is at most 4.
In this section, we will compute the Goeritz group of $M = H_1 \cup_\Sigma H_2$ in two cases. 
The first case, where we will consider in Subsection~\ref{subsec:The Goeritz groups of distance-2 Heegaard splittings}, 
is that the gluing map $\varphi$ is particularly simple in the sense that $\varphi$ is 
a power of the Dehn twist about the binding $b$. 
The second case, where we will consider in Subsection~\ref{subsec:The Goeritz groups of distance-4 Heegaard splittings}, 
is, on the contrary, that the gluing map $\varphi$ is complicated in the sense that  
the distance in $\mathcal{C} (\Sigma_b) $ between the images of subsurface projection $\pi_{\Sigma_b}$ of 
$\mathcal{D} (H_1)$ and $\mathcal{D} (H_2)$ is sufficiently large, 
where $\Sigma_b := \Cl (\Sigma - \Nbd (b))$.

\subsection{The Goeritz groups of distance-2 Heegaard splittings}
\label{subsec:The Goeritz groups of distance-2 Heegaard splittings}


Let $k$ be an integer. 
Suppose that the gluing map  
$\partial H_1 \to \partial H_2$ is the $k$-th power $\tau_b^k$ of the Dehn twist $\tau_b$. 
Note that by  Lemma 3.1 and Yoshizawa \cite[Theorem 1.3]{Yos14}, if $|k| \geq 2$ 
the distance $d (H_1, H_2)$ of this splitting is exactly $2$. 
The aim of this subsection is to prove the following theorem. 

\begin{theorem}
\label{thm:the Goeritz groups of distance-2 Heegaard splittings}
Suppose that $|k| \geq 5$. 
For the Heegaard splitting $M = H_1 \cup_\Sigma H_2$ as above, we have the following.
\begin{enumerate}
\item
The splitting $M = H_1 \cup_\Sigma H_2$ is not induced from an open book decomposition. 
\item
The Goeritz group $\MCG_+ (M , H_1)$ is isomorphic to the group $\MCG (F)$. 
In particular, $\MCG_+ (M , H_1)$ is an infinite group. 
\end{enumerate} 
\end{theorem}

\begin{proof}[Proof of Theorem~$\ref{thm:the Goeritz groups of distance-2 Heegaard splittings}$ $(1)$]
We suppose for a contradiction that the Heegaard splitting $M = H_1 \cup_{\Sigma} H_2$ is induced from 
an open book decomposition. 
Let $b'$ be the binding of the open book decomposition. 
Using the identification of $H_1$ with $H$, we regard $b$ and $b'$ 
as bindings of $H$. 
Since $\tau_b^k$ is the gluing map for the Heegaard splitting, 
$\tau_b^k(b')$ is a binding of $H$ as well. 
By Lemma~\ref{lem:the binding is 1-seamed} there exists 
a solid pants decomposition $\mathscr{S}$ of $H$ such that 
$b$ is $1$-seamed with respect to $\mathscr{S}$. 
Since $\Sigma - b$ is connected whereas $\Sigma - b'$ consists of two components, 
$b$ and $b'$ are not isotopic on $\Sigma$.  

Suppose first that $b \cap b' = \emptyset$. 
Let $\pi' : H \to F'$ be the $I$-bundle with $b'$ the binding. 
Needless to say, this is the trivial bundle. 
Hence, $\pi' (b)$ is a simple closed curve on $F'$. 
Since $b$ and $b'$ are not parallel, and $F'$ is orientable, 
there exists an essential simple arc $\alpha$ on $F'$ disjoint from $\pi' (b)$. 
Then ${\pi'}^{-1} (\alpha)$ is an essential disk in $H$ disjoint from $b$.  
It follows that $d_{\mathcal{C} (\Sigma)} (b, \mathcal{D} (H) ) \leq 1$. 
This contradicts Lemma~\ref{lem:distance of the binding}. 

Suppose that $b \cap b' \neq \emptyset$. 
By Lemma~\ref{lem:obstruction to be a binding}
the binding $b'$ cannot be $2$-seamed with respect to $\mathscr{S}$. 
Let $\mathscr{P}$ be the set of the boundaries of the disks in $\mathscr{S}$. 
Let $C$ be the union of the simple closed curves of $\mathscr{P}$. 
Set $N := \Nbd (b; \Sigma)$. 
 We may isotope $b'$ so that $b'$ in efficient position with respect to $(N, C)$ and 
 $(N, \mathscr{P}, b')$ admits no outer triangles. 
Let $\alpha_1, \ldots, \alpha_r$ be the components of $b' \cap N$, and 
$t_j$ $(j \in \{1 , \ldots , r \})$ the twisting number of $\alpha_j$ in $N$ with respect to $C$. 
By Lemma~\ref{lem:lower bound for the number of seams} 
we have $|t_j| \leq 2$ for all $j$. 
Since $|k| \geq 5$ by the assumption, this implies that 
either $ k + t_j \geq 0$ $($for all $j)$ or 
$ k + t_j \leq 0$ $($for all $j)$. 
It then follows from Lemma~\ref{lem:minimality of twisting number and outer triangles} 
that 
$\tau_b^k ( b' )$ remains to be in efficient position with respect to $(N , C)$, 
and the twisting number of $\tau_b^k (\alpha_j)$ in $N$ with respect to $N$ is $k + t_j$. 
In particular, we have $|k + t_j|  \geq |k| - |t_j|  \geq |k| - 2 > 2$. 
Again by Lemma~\ref{lem:lower bound for the number of seams}, the binding $\tau_b^k (b')$ is 
$2$-seamed with respect to $\mathscr{S}$. 
 This contradicts Lemma~\ref{lem:obstruction to be a binding}. 
\end{proof}

To prove Theorem~$\ref{thm:the Goeritz groups of distance-2 Heegaard splittings}$ $(2)$, 
we need the following lemma. 
\begin{lemma}
\label{lem:extending automorphisms of Sigma to M}
Let $\varphi$ be an automorphism of $\Sigma$ that is extendable over $H_1$. 
If $\varphi$ preserves the binding $b$, $\varphi$ is extendable over $H_2$ as well. 
Thus, $\varphi$ can be regarded as an element of $\MCG_+ (M, H_1)$. 
\end{lemma}
\begin{proof}
We will first show that $\varphi$ commutes with $\tau_b^k$ up to isotopy. 
We identify $\Nbd (b ; \Sigma )$ with $S^1 \times I$. 
Let $R$ and $T_k$ be the automorphisms of $S^1 \times I$
defined by $R (\theta, r) = (- \theta , 1-r )$ and 
$T_k (\theta, r) = (\theta + 2 \pi k r, r )$. 
Clearly $R$ commutes with $T_k$. 
Up to isotopy, we can assume that $\varphi$ preserves $\Nbd (b; \Sigma)$ and 
$\varphi |_{\Nbd (b; \Sigma)}$ is the identity or $R$. 
We can also assume that the support of $\tau_b^k$ is $\Nbd (b; \Sigma)$ and 
$\tau_b^k |_{\Nbd (b; \Sigma)} = T_k$. 
Therefore $\varphi$ commutes with $\tau_b^k$ up to isotopy. 

To prove that $\varphi$ is extendable over $H_2$, it suffices to see that 
$\varphi (\mathcal{D} (H_2)) = \mathcal{D} (H_2)$. 
This is equivalent to say that $\varphi ( \tau_b^k (\mathcal{D} (H_1))) = \tau_b^k (\mathcal{D} (H_1))$. 
Since $\varphi$ is extendable over $H_1$, it holds  
$\varphi (\mathcal{D} (H_1)) = \mathcal{D} (H_1)$. 
Therefore it follows that 
$\varphi ( \tau_b^k (\mathcal{D} (H_1))) = \tau_b^k ( \varphi (\mathcal{D} (H_1))) = \tau_b^k (\mathcal{D} (H_1))$. 
\end{proof}

Recall that $F$ is a compact non-orientable surface with $\chi (F) < 0$ and $\# \partial F = 1$, 
and  $\pi: H \to F$ is the orientable $I$-bundle with the binding $b$. 
We regard that $F \subset H$ with $\partial F = b$.  
The annulus $\pi^{-1} (\partial F) = \Nbd (b)$ is equipped with the structure of a subbundle of 
$\pi: H \to F$. 
The restriction of $ \pi$ to $\Sigma_b$ $(= \Cl (\Sigma - \Nbd (b)))$ is the orientation double cover of $F$. 
Using the identification of $H_1$ with $H$, we regard $F$ as a surface in $H_1$. 
By Lemma~\ref{lem:injectivity of the map from MCG(F) to MCG+(Sigma)},  
each element $\varphi_F \in \MCG (F)$ lifts to a unique element of $\Aut_+ (\Sigma_b)$. 
Using the $I$-bundle structure of $\Nbd (b)$, this element extends to an automorphism of $\Sigma$ in a unique way. 
Clearly, this is extendable over $H_1$, and further, extendable over $H_2$ as well by 
Lemma~\ref{lem:extending automorphisms of Sigma to M}. 
In this way we get a map $L : \MCG (F) \to \MCG_+ (M, H_1)$.

\begin{proof}[Proof of Theorem~$\ref{thm:the Goeritz groups of distance-2 Heegaard splittings}$ $(2)$]
We will show below that the above map $L : \MCG (F) \to \MCG_+ (M, H_1)$ is an isomorphism. 
The injectivity immediately follows from Lemma~\ref{lem:injectivity of the map from MCG(F) to MCG+(Sigma)}. 
To prove the surjectivity of $L$, it suffices to see that any map $\varphi \in \MCG_+ (M, H_1)$ preserves 
the binding $b$ (up to isotopy). 
Indeed, there exists a unique $I$-bundle structure of $H$ with $b$ the binding. 
Thus, if $\varphi$ preserves $b$ (up to isotopy), it preserves $F$ (up to isotopy). 
Putting $\varphi_F := \varphi|_F$, we have $\varphi = L (\varphi_F)$. 
Suppose for a contradiction that 
there exists a map $\varphi \in \MCG_+ (M, H_1)$ that does not preserve $b$. 

First we will show that we can replace $\varphi$ with another one, if necessary,  
so that  $b \cap \varphi (b) \neq \emptyset$. 
Suppose that $b \cap \varphi (b) = \emptyset$. 
Then $\varphi (b)$ is a simple closed curve on $ \Sigma_b := \Cl (\Sigma - \Nbd (b))$. 
Let $\alpha$ and $\beta$ be two-sided simple closed curves on $F$ satisfying 
$d_{\mathcal{C}(F)} (\alpha , \beta) \geq 3$. 
Due to Penner \cite{Pen88}, the composition $\tau_\alpha \circ \tau_\beta$ of Dehn twists 
is pseudo-Anosov. 
Let $\psi$ be the element of $\Aut_+ (\Sigma)$ defined by taking an orientation-preserving 
lift of $\tau_\alpha \circ \tau_\beta$ to $\Aut_+ (\Sigma_b)$, and then extending it to 
the automorphism of the whole $\Sigma$ 
as explained right before the proof. 
Note that $\psi |_{\Sigma_b}$ is also a pseudo-Anosov map. 
Thus, for a sufficiently large integer $n$, we have 
$\psi^n (\varphi (b)) \cap \varphi (b) \neq \emptyset$. 
By Lemma~\ref{lem:extending automorphisms of Sigma to M}, 
$\psi$ can be regarded as an element of the Goeritz group $\MCG_+ (M , H_1)$. 
Therefore, $\varphi^{-1} \circ \psi^n \circ \varphi$ is an element of  $\MCG_+ (M , H_1)$ 
that satisfies $(\varphi^{-1} \circ \psi^n \circ \varphi) (b) \cap b \neq \emptyset$. 
 
In the following, we assume that $b \cap \varphi (b) \neq \emptyset$. 
Set $b' := \varphi (b)$. 
Since $b$ is a binding of a twisted book decomposition of $M$, so is $b'$ 
of another twisted book decomposition of $M$ that induces that same Heegaard splitting 
$M = H_1 \cup_\Sigma H_2$. 
As explained in the proof of 
Theorem~$\ref{thm:the Goeritz groups of distance-2 Heegaard splittings}$ $(1)$, 
it follows that both $b'$ and $\tau_b^k (b')$ are bindings of $H_1$. 
The same argument as in the proof of 
Theorem~$\ref{thm:the Goeritz groups of distance-2 Heegaard splittings}$ $(1)$ 
shows that at least one of $b'$ and $\tau_b^k (b')$ is $2$-seamed with respect to 
a solid pants decomposition $\mathscr{S}$ of $H_1$. 
Thus, by Lemma~\ref{lem:obstruction to be a binding} at least one of $b'$ and $\tau_b^k (b')$ 
is not a binding of $H_1$. 
This is a contradiction.  
\end{proof}

\subsection{The Goeritz groups of distance-4 Heegaard splittings}
\label{subsec:The Goeritz groups of distance-4 Heegaard splittings}

Recall that $H_1$ and $H_2$ are copies of $H$, and 
$M = H_1 \cup_\Sigma H_2$ is the Heegaard splitting with the gluing map $\varphi \in \Aut_+ (\Sigma , b)$. 
Let $(S, \iota_0, \iota_1)$ be the twisted book decomposition of $M$ that induces 
$M = H_1 \cup_\Sigma H_2$.  
Set $G := \MCG (S)$ and $G_+ := \MCG_+ (S)$. 
Let $G(S, \iota_0, \iota_1)$ denote the intersection of 
the centralizers $C_G (\iota_0)$, $C_G (\iota_1)$, and the subgroup $G_+$ of $G$, 
that is, $G(S, \iota_0, \iota_1) = C_G (\iota_0) \cap C_G (\iota_1) \cap G_+$. 
Set $\mathcal{D}_{\Sigma_b} := \pi_{\Sigma_b}  (\mathcal{D} (H))$. 
Also, recall that $\Sigma_b = \Cl (\Sigma - \Nbd (b))$ 
and $\pi_{\Sigma_b} : 
\mathcal{C}^{(0)} ( \partial H ) \to P (\mathcal{C}^{(0)} ( \Sigma_b ) )$ is a subsurface projection. 

The following is the main theorem of this subsection: 
\begin{theorem}
\label{thm:the Goeritz groups of distance-4 Heegaard splittings}
Suppose $d_{\mathcal{C}(\Sigma_b)} ( \mathcal{D}_{\Sigma_b} , \varphi (\mathcal{D}_{\Sigma_b}) ) > 10$. 
For the Heegaard splitting $M = H_1 \cup_\Sigma H_2$ as above, we have the following.
\begin{enumerate}
\item
The distance $d (H_1, H_2)$ is exactly $4$. 
\item
The Goeritz group $\MCG_+ (M , H_1)$ is isomorphic to the group $G (S, \iota_0, \iota_1)$. 
\end{enumerate} 
\end{theorem}

In Lemma~\ref{lem:exisgence of f that take D far from D}, we will see that 
there actually exists a Heegaard splitting satisfying the condition in 
Theorem~\ref{thm:the Goeritz groups of distance-4 Heegaard splittings}. 

Recall that $F$ is a compact non-orientable surface with $\chi (F) < 0$ and $\# \partial F = 1$, and 
$\pi: H \to F$ is the orientable $I$-bundle with the binding $b$.  
We equip with $\Int \, F$ and $\Int \, \Sigma_b$ hyperbolic metrics so that 
the covering map $p := \pi|_{\Int \Sigma_b}$ is a local isometry. 
Consider the pull-back $p^{*} : \ML (F) \to \ML (\Sigma_b)$ 
defined by $p^{*}(\lambda, m) = ( p^{-1} (\lambda) , m \circ p )$ 
for $(\lambda , m) \in \ML (F)$. 
Clearly, this is a well-defined, injective  
piecewise linear map that is equivariant under the action of $\Real_+$. 
Thus, this map induces an injective piecewise projective map 
$c : \mathbb{P} \ML (F) \to \mathbb{P} \ML (\Sigma_b)$. 
Let $\mathcal{F} \subset \mathbb{P} \ML ( \Sigma_b)$ 
denote the image of the set $\mathcal{S} (F)$ of weighted multicurves on $F$ by the map $c$. 

\begin{lemma}[Johnson \cite{Joh11c}]
\label{lem:F is nowhere dense in PML}
The set $\mathcal{F}$ is nowhere dense in $\mathbb{P} \ML (\Sigma_b)$. 
\end{lemma}
In the unpublished paper \cite{Joh11c}, Johnson gave a sketch of the proof of this lemma. 
The following proof is essentially due to his idea. 

\begin{proof}[Proof of Lemma~$\ref{lem:F is nowhere dense in PML}$]
By Lemma~\ref{thm:multicurves in ML}, 
the set $\mathcal{S}(F)$ is dense in $\mathbb{P}\ML (F) $.  
Since $c$ is a continuous map between spheres, which are compact and Hausdorff, we have 
\[ c ( \mathbb{P} \ML (F) ) = c ( \Cl ( \mathcal{S} (F) ) )  =  \Cl ( c ( \mathcal{S} (F)) ) = \Cl (  \mathcal{F}  ) . \]

Let $F = \#_h \mathbb{R}P^2 - \Int (D^2)$ (thus, $\Sigma_b = \#_{h-1} T^2 - \sqcup_{2} \Int (D^2)$). 
By Theorem~\ref{thm:PL structure on ML}, we have 
$\mathbb{P} \ML (F) \cong S^{3 h - 5}$ and 
$\mathbb{P} \ML (\Sigma_b) \cong S^{6h - 9}$. 
Thus, $c$ is a piecewise projective embedding. 
Noting that $3h - 5 < 6h - 9$ for $h \geq 2$, 
we conclude that $\Image \, c = \Cl (\mathcal{F})$ 
is nowhere dense in $\mathbb{P} \ML (\Sigma_b)$. 
\end{proof}

Let $\mathcal{I}$ denote the set of projectivizations of stable and unstable laminations of 
pseudo-Anosov automorphisms of $\Sigma_b$.  
In the following, by abuse of notation we simply write $\lambda$ to mean a projective geodesic measured lamination 
$[(\lambda, m)] \in \mathbb{P} \ML (\Sigma_b)$. 
This will not cause any confusion.  

\begin{lemma}
\label{lem:invariant laminations}
\begin{enumerate}
\item
The set $\mathcal{I}$ is dense in $\mathbb{P} \ML (\Sigma_b)$. 
\item
Let $\lambda$ be a point of $\mathcal{I}$, and  $\lambda'$ a point of $\mathbb{P} \ML (\Sigma_b)$. 
If the intersection number of any representatives of $\lambda$ and $\lambda'$ 
in $\ML (\Sigma_b)$ is zero, then $\lambda = \lambda'$. 
\end{enumerate}
\end{lemma}
\begin{proof}
(1) follows from Long \cite[Lemma 2.6]{Lon86}. 
(2) follows from a well-known fact that 
the stable and unstable laminations for a pseudo-Anosov automorphism are 
minimal, uniquely ergodic, and fill up the surface. 
\end{proof}

\begin{lemma}
\label{lem:DSigma is nowhere dense in PML}
The set $\mathcal{D}_{\Sigma_b}$ is nowhere dense in $\mathbb{P} \ML (\Sigma_b)$. 
\end{lemma}
\begin{remark}
It it worth noting that in \cite{Mas86} Masur proved that 
$\mathcal{D}(H)$ is nowhere dense in $\mathbb{P} \ML (\Sigma)$.  
\end{remark}
\begin{proof}
Suppose for a contradiction that $\mathcal{D}_{\Sigma_b}$ is not nowhere dense, 
that is, there exists an open set $U$ of $\mathbb{P} \ML (\Sigma_b)$ contained in 
$\Cl (\mathcal{D}_{\Sigma_b})$. 
We will prove that this implies that  $U$ is also contained in $\Cl (\mathcal{F})$, which contradicts 
Lemma~\ref{lem:F is nowhere dense in PML}. 

To prove that, we show that the set $U \cap \mathcal{I}$ is contained in $\Cl (\mathcal{F})$. 
Let $\lambda \in U \cap \mathcal{I}$. 
Since $U$ is contained in $\Cl ( \mathcal{D}_{\Sigma_b} )$, 
there exists a sequence $(\alpha_n)$ in $\Cl (\mathcal{D}_{\Sigma_b})$ such that $\alpha_n$ 
converges to $\lambda$ as $n$ tends to $ \infty$.  
For each $\alpha_n$, we have $d_{\mathcal{C} (\Sigma_b)} (\alpha_n , \mathcal{F}) \leq 3$ due to 
Masur-Schleimer \cite[Lemma 12.20]{MS13}. 
Thus, for each $n$ there exists a path $(\beta_n^0, \beta_n^1, \beta_n^2, \beta_n^3)$ such that 
$\beta_n^0 = \alpha_n$ and $\beta_n^3 \in \mathcal{F}$. 
By Theorem~\ref{thm:PL structure on ML}, 
$\mathbb{P} \ML (\Sigma_b )$ is sequentially compact. 
After passing to a subsequence if necessary, which we still write $(\beta_n^j)$, 
we can assume that the sequence 
$(\beta_n^j)$ converges to a point $\lambda^j$ in $\mathbb{P} \ML (\Sigma_b)$ (as $n \to \infty$) for 
all $j \in \{  0 , 1 , 2 , 3 \}$. 
Note that $\lambda^0 = \lambda \in \mathcal{I}$ and $\lambda^3 \in \Cl  ( \mathcal{F} )$. 
Since the intersection number of any representatives of $\beta_n^j$ and $\beta_n^{j+1}$ in 
$\ML (\Sigma_b)$ is zero for all $n$ and $j$, 
that of any representatives of $\lambda^j$ and $\lambda^{j+1}$ in $\ML (\Sigma_b)$ is zero for 
all $j \in \{  0 , 1 , 2 \}$. 
Since $\lambda^0 \in \mathcal{I}$, we have $\lambda^0 = \lambda^1$ by Lemma~\ref{lem:invariant laminations} (2). 
Applying the same argument repeatedly, we finally get $\lambda^0 = \cdots = \lambda^3$. 
Therefore, $\lambda$ is contained in $\Cl (\mathcal{F})$. 

By Lemma~\ref{lem:invariant laminations} (1), the set $\mathcal{I}$ is dense in $\mathbb{P} \ML (\Sigma_b)$. 
Thus, we conclude that $U \subset \Cl (U \cap \mathcal{I}) = \Cl (\mathcal{F})$. 
\end{proof}

The following lemma shows the existence of a Heegaard splitting satisfying the condition in 
Theorem~\ref{thm:the Goeritz groups of distance-4 Heegaard splittings}. 
\begin{lemma}
\label{lem:exisgence of f that take D far from D}
There exists an automorphism $\psi$ of $\Sigma_b$ such that 
$d_{\mathcal{C}(\Sigma_b)} ( \mathcal{D}_{\Sigma_b} , \psi^n (\mathcal{D}_{\Sigma_b}) )$ 
tends to $\infty$ as $n$ tends to $\infty$. 
\end{lemma}
\begin{proof}
By Lemma~\ref{lem:DSigma is nowhere dense in PML}, 
$\mathcal{D}_{\Sigma_b}$ is nowhere dense in $\mathbb{P} \ML (\Sigma_b)$. 
Since $\mathcal{I}$ is dense in $\mathbb{P} \ML (\Sigma_b)$ by Lemma~\ref{lem:invariant laminations}, there exists a pseudo-Anosov automorphism 
$\psi: \Sigma_b \to \Sigma_b$ such that none of its invariant laminations 
$\lambda^+$, $\lambda^-$ is contained in $\Cl ( \mathcal{D}_{\Sigma_b} )$. 
We will show that $\psi$ is the required automorphism in the assertion. 

Suppose for a contradiction that there exists $N > 0$ such that 
$d_{\mathcal{C} (\Sigma_b)} ( \mathcal{D}_{\Sigma_b} , \psi^n (\mathcal{D}_{\Sigma_b})) \leq N$ 
for any $n > 0$.  
Thus, for each $n$ there exists a path $(\alpha_n^0, \ldots , \alpha_n^N)$ such that 
$\alpha_n^0 \in \mathcal{D}_{\Sigma_b}$ and $\alpha_n^N \in \psi^n (\mathcal{D}_{\Sigma_b})$. 
Recall that $\mathbb{P} \ML (\Sigma_b)$ is sequentially compact by Theorem~\ref{thm:PL structure on ML}. 
After passing to a subsequence if necessary, which we still write $(\alpha_n^j)$, 
we can assume that the sequence $(\alpha_n^j)$ converges to a point 
$\lambda^j$ in $\mathbb{P} \ML (\Sigma_b)$ (as $n \to \infty$) for 
all $j \in \{  0 , \ldots , N \}$. 
Note that $\lambda^0 \in \Cl (\mathcal{D}_{\Sigma_b})$ and $\lambda^N = \lambda^- $. 
Since the intersection number of any representatives of $\alpha_n^j$ and $\alpha_n^{j+1}$ 
in $\ML (\Sigma_b)$ is zero for all $n$ and $j$, 
that of any representatives of $\lambda^j$ and $\lambda^{j+1}$ in $\ML (\Sigma_b)$ is zero for 
all $j \in \{  0 , \ldots , N-1 \}$. 
Since $\lambda^N \in \mathcal{I}$, we have $\lambda^{N-1} = \lambda^N$ 
by Lemma~\ref{lem:invariant laminations} (2). 
Applying the same argument repeatedly, we finally get $\lambda^0 = \cdots = \lambda^N$. 
This is impossible because $\lambda^0 \in \Cl (\mathcal{D}_{\Sigma_b})$ and 
$\lambda^N = \lambda^- \not\in \Cl (\mathcal{D}_{\Sigma_b})$. 
\end{proof}

\begin{lemma}
\label{lem:distance between D(H) and f(D(H))}
Let $\psi$ be an automorphism of $\Sigma$ that preserve the binding $b$. 
If the distance $d_{\mathcal{C}(\Sigma_b)} ( \mathcal{D}_{\Sigma_b} , \psi (\mathcal{D}_{\Sigma_b}) )$ 
is greater than $6$, 
the distance $d_{\mathcal{C}(\Sigma)} ( \mathcal{D}(H) , \psi (\mathcal{D}(H))) $ 
is exactly $4$. 
\end{lemma}
\begin{proof}
 By Lemma~\ref{lem:distance of the binding}, 
 the distance $d_{\mathcal{C}(\Sigma)} ( \mathcal{D}(H) , \psi (\mathcal{D}(H))) $ 
is at most $4$ for any $\psi$. 
Suppose that 
 $d_{\mathcal{C}(\Sigma_b)} ( \mathcal{D}_{\Sigma_b} , \psi (\mathcal{D}_{\Sigma_b}) ) > 6$. 
Suppose for a contradiction that 
the distance $d_{\mathcal{C}(\Sigma)} ( \mathcal{D}(H) , \psi (\mathcal{D}(H))) $ is less than $4$. 
Then there exists an integer $k \in \{ 0 , 1 , 2 , 3\}$ and a geodesic segment $(\alpha_0 , \ldots , \alpha_k )$ 
in $\mathcal{C} (\Sigma)$ with 
$\alpha_0 \in \mathcal{D} (H)$ and $\alpha_k \in \psi ( \mathcal{D} (H) )$. 
If there exists $j$ such that $\alpha_j = b$, we have 
either $d_{C (\Sigma)} ( \mathcal{D} (H) , b ) < 2$ or 
$d_{\mathcal{C}(\Sigma)} ( b , \psi ( \mathcal{D} (H) ) ) < 2$. 
Since $b$ is a binding of the handlebody whose disk sets corresponds to 
$\psi ( \mathcal{D} (H) )$, this is impossible by Lemma~\ref{lem:distance of the binding}. 
Suppose that $\alpha_j \neq b$ for all $j$. 
Then by Lemma~\ref{lem:subsurface projection of a geodesic segment}, 
we have $d_{\mathcal{C}(\Sigma_b)} ( \mathcal{D}_{\Sigma_b} , \psi (\mathcal{D}_{\Sigma_b}) ) \leq  6$, 
which is a contradiction. 
\end{proof}
We remark that in the proof of Lemma~\ref{lem:distance between D(H) and f(D(H))} we have used the assumption that $F$ has a single 
boundary component. 
Indeed, Lemma~\ref{lem:distance of the binding}, which is used to get a 
contradiction in the argument, is valid only when 
the binding $b$ on $\partial H$ is a single simple closed curve. 
In the case where $b$ has more than one components, we have 
$d_{\mathcal{C} (\Sigma)} (b , \mathcal{D} (H)) = 1$, which cannot lead to a contradiction.

\begin{proof}[Proof of Theorem~$\ref{thm:the Goeritz groups of distance-4 Heegaard splittings}$]
The first assertion is a direct consequence of 
Lemma~\ref{lem:distance between D(H) and f(D(H))}. 

By identifying $\Sigma_b$ with $S \times \{ 1/2 \}$, we get a natural injective homomorphism 
$\eta : G (S, \iota_0, \iota_1) \to \MCG_+ (M, H_1)$. 
We will show the surjectivity of $\eta$. 
Suppose that there exists an element $\varphi \in \MCG_+ (M, H_1)$ 
such that $\varphi (b) \neq b$. 
Set $b' := \varphi (b)$. 
Since $b'$ is also a binding of a twisted book decomposition of $M$, 
we have $d_{\mathcal{C} (\Sigma)} ( b' , \mathcal{D} (H_j)) = 2$ for $j \in \{ 1 , 2 \}$ 
by Lemma~\ref{lem:distance of the binding}. 
Set $\mathcal{D}_{\Sigma_b}^j := \pi_{\Sigma_b} (\mathcal{D} (H_j))$. 
By Lemma~\ref{lem:subsurface projection of a geodesic segment}, 
we have $d_{\mathcal{C} (\Sigma_b)} ( \pi_{\Sigma_b} (b') , \mathcal{D}_{\Sigma_b}^j) \leq 4$.
This together with the fact that 
the diameter of $\pi_{\Sigma_b} (b')$ is at most $2$ implies that
$d_{\mathcal{C} (\Sigma_b)} (\mathcal{D}_{\Sigma_b}^1 , \mathcal{D}_{\Sigma_b}^2)$ 
is at most $10$. 
This contradicts the assumption on $\varphi$. 
In Consequence, any element of $\MCG_+ (M, H_+)$ preserves the binding $b$. 

Let $\varphi \in \MCG_+ (M, H_1)$. 
Let $q : S \times I \to M$ be the quotient map. 
Set $S_t := q (S \times \{ t \} )$ for $t \in [0, 1]$. 
Since the $I$-bundle structure of $H$ with the binding $b$ is unique, 
$\varphi$ preserves each $S_t$. 
In particular, $\varphi $ restricts to an orientation preserving 
automorphism of $S_{1/2} = \Sigma$. 
Thus $\varphi$ is contained in the image of $\eta$. 
\end{proof}

\section*{Acknowledgments} 
The authors would like to thank the anonymous referee
for his or her valuable comments and suggestions 
that helped them to improve the exposition. 
The second author is supported by JSPS KAKENHI Grant
 Numbers JP15H03620, JP17K05254, JP17H06463, and JST CREST Grant Number JPMJCR17J4.


\begin{thebibliography}{99999}

\bibitem{Akb08}
Akbas, E.,
A presentation for the automorphisms of the $3$-sphere that
preserve a genus two Heegaard splitting,
Pacific J. Math. \textbf{236} (2008), no. 2, 201--222.

\bibitem{Cho08} 
Cho, S.,
Homeomorphisms of the $3$-sphere that preserve a Heegaard splitting of genus two,
Proc. Amer. Math. Soc. \textbf{136} (2008), no. 3, 1113--1123.

\bibitem{Cho13} 
Cho, S.,
Genus two Goeritz groups of lens spaces,
Pacific J. Math.  \textbf{265}  (2013),  no. 1, 1--16.

\bibitem{CK14} 
Cho, S., Koda, Y.,
The genus two Goeritz group of $\mathbb S^2 \times \mathbb S^1$,
Math. Res. Lett.  \textbf{21}  (2014),  no. 3, 449--460.

\bibitem{CK15} 
Cho, S., Koda, Y.,
Disk complexes and genus two Heegaard splittings for non-prime 3-manifolds,
Int. Math. Res. Not. IMRN
{\bf 2015} (2015), 4344--4371.

\bibitem{CK16} 
Cho, S., Koda, Y.,
Connected primitive disk complexes and genus two Goeritz groups of lens spaces,
Int. Math. Res. Not. IMRN
{\bf 2016} (2016), 7302-7340

\bibitem{CK19}
Cho, S., Koda, Y., 
The mapping class groups of reducible Heegaard splittings of genus two, 
Transactions of the American Mathematical Society {\bf 371} (2019), no. 4, 2473--2502. 

\bibitem{CKS16} 
Cho, S., Koda, Y., Seo, A.,
Arc complexes, sphere complexes and Goeritz groups,
Michigan Math. J.  \textbf{65}  (2016),  no. 2, 333--351.

\bibitem{Eps66}
Epstein, D. B. A., Curves on $2$-manifolds and isotopies, Acta Math. {\bf 115} (1966), 83--107.

\bibitem{FLP79}
Fathi, A., Laudenbach, F., Po\'{e}naru, V., 
Travaux de Thurston sur les surfaces, S\'{e}minaire Orsay, Ast\'{e}risque, 66-67, 
Soci\'{e}t\'{e} Math\'{e}matique de France, Paris, 1979, 284 pp.

\bibitem{FS18}
Freedman, M., Scharlemann, M., 
Powell moves and the Goeritz group, arXiv:1804.05909. 

\bibitem{Goe33}
Goeritz, L.,
Die Abbildungen der Berzelfl\"{a}che und der Volbrezel vom Gesschlect $2$,
Abh. Math. Sem. Univ. Hamburg \textbf{9} (1933), 244--259.

\bibitem{Gor07}
Gordon, C.,
Problems. {\it Workshop on Heegaard Splittings}, pp. 401--411, 
Geom. Topol. Monogr. {\bf 12}, Geom. Topol. Publ., Coventry, 2007. 

\bibitem{Hem01}
Hempel, J., 
$3$-manifolds as viewed from the curve complex, Topology {\bf 40} (2001), no. 3, 631--657.

\bibitem{IJK18}
Ido, A., Jang, Y., Kobayashi, T., 
On keen Heegaard splittings, Singularities in generic geometry, 
293--311, Adv. Stud. Pure Math., 78, Math. Soc. Japan, Tokyo, 2018.

\bibitem{Joh10}
Johnson, J., 
Mapping class groups of medium distance Heegaard splittings, 
Proc. Amer. Math. Soc. {\bf 138} (2010), no. 12, 4529--4535.


\bibitem{Joh11a}
Johnson, J.,
Automorphisms of the three-torus preserving a genus-three Heegaard splitting,
Pacific J. Math.  {\bf 253}  (2011),  no. 1, 75--94.

\bibitem{Joh11b}
Johnson, J., 
Mapping class groups of once-stabilized Heegaard splittings, 
arXiv:1108.5302. 

\bibitem{Joh11c}
Johnson, J., 
One-sided and two-sided Heegaard splittings, 
arXiv:1112.0471. 

\bibitem{Joh11d}
Johnson, J., 
Heegaard splittings and open books, 
arXiv:1110.2142. 



\bibitem{JR13}
Johnson, J., Rubinstein, H., 
Mapping class groups of Heegaard splittings, 
J. Knot Theory Ramifications {\bf 22} (2013), No. 5, 1350018, 20 pp.

\bibitem{Lon86}
Long, D. D., A note on the normal subgroups of mapping class groups, Math. Proc.
Cambridge Philos. Soc. {\bf 99} (1986), no. 1, 79--87.

\bibitem{McC06}
McCullough, D., Homeomorphisms which are Dehn twists on the boundary, Algebr.
Geom. Topol. {\bf 6} (2006), 1331--1340.

\bibitem{Mas86}
Masur, H. A., Measured foliations and handlebodies, 
Ergodic Theory Dynam. Systems {\bf 6} (1986), no. 1, 99--116.

\bibitem{MM00}
Masur, H. A., Minsky, Y. N., 
Geometry of the complex of curves II: Hierarchical structure, Geom. Funct. Anal. 10 (2000), no. 4, 902--974.

\bibitem{MS13}
Masur, H., Schleimer, S., The geometry of the disk complex, J. Amer. Math. Soc. {\bf 26} 
(2013), no. 1, 1--62.

\bibitem{Nam07}
Namazi, H., Big Heegaard distance implies finite mapping class group, Topology Appl.
{\bf 154} (2007), 2939--2949.

\bibitem{Oer02}
Oertel, U., Automorphisms of three-dimensional handlebodies, Topology {\bf 41} (2002), no.
2, 363--410.

\bibitem{Pen88}
Penner, R. C., A construction of pseudo-Anosov homeomorphisms, Trans. Amer. Math.
Soc. {\bf 310} (1988), no. 1, 179--197.

\bibitem{PH92}
Penner, R. C., Harer, J. L., Combinatorics of train tracks, Annals of Mathematics
Studies, 125, Princeton University Press, Princeton, NJ, 1992.

\bibitem{Ree81}
Rees, M., 
An alternative approach to the ergodic theory of measured foliations on surfaces, 
Ergodic Theory Dynamical Systems {\bf 1} (1981), no. 4, 461--488. 

\bibitem{Sch04} 
Scharlemann, M.,
Automorphisms of the $3$-sphere that preserve a genus two Heegaard splitting,
Bol. Soc. Mat. Mexicana (3) \textbf{10} (2004), Special Issue, 503--514.

\bibitem{Sc13} 
Scharlemann, M.,
Generating the genus $g+1$ Goeritz group of a genus $g$ handlebody,
Geometry and topology down under,  347--369, Contemp. Math., 597, Amer. Math. Soc., Providence, RI, 2013.


\bibitem{Ser60}
Serre, J. P., Rigidit\'{e} du foncteur de Jacobi d'\'{e}chelon $n \geq 3$, Appendice d'expos\'{e} 17, S\'{e}minaire Henri Cartan 13e ann\'{e}e, 1960/61.

\bibitem{Thu88} 
Thurston, W. P.,  
\emph{The Geometry and Topology of Three-Manifolds}, 
MSRI, 2002, http://www.msri.org/publications/books/gt3m/. 

\bibitem{Thu88} 
Thurston, W. P., 
On the geometry and dynamics of diffeomorphisms of surfaces, 
Bull. Amer. Math. Soc. (N.S.) {\bf 19} (1988), no. 2, 417--431.

\bibitem{Yos14}
Yoshizawa, M., High distance Heegaard splittings via Dehn twists, Algebr. Geom. Topol. {\bf 14} (2014), no. 2, 979--1004.
\end{thebibliography}
\end{document}